\documentclass[11pt,a4paper]{article}
\usepackage[top=3cm, bottom=3cm, left=3cm, right=3cm]{geometry}
\usepackage{tabu}
\usepackage{breqn}
\usepackage[hang,flushmargin]{footmisc} 

\usepackage{amsmath,amsthm,amssymb,graphicx, multicol, array}
\usepackage{enumerate}
\usepackage{enumitem}
\usepackage{hyperref}
\usepackage{setspace}
\usepackage{xcolor}
\usepackage{currfile}

\usepackage{tabto}

\DeclareMathOperator{\ord}{ord}

\DeclareMathOperator{\tmod}{mod}

\newtheorem{thm}{Theorem}[section]
\newtheorem{lem}{Lemma}[section]
\newtheorem{conj}{Conjecture}[section]

\newtheorem{cor}{Corollary}[section]
\newtheorem{dfn}{Definition}[section]
\newtheorem{exe}{Exercise}[section]

\newcommand{\Z}{\mathbb{Z}}

\newcommand{\R}{\mathbb{R}}
\newcommand{\C}{\mathbb{C}}
\newcommand{\F}{\mathbb{F}}

\usepackage{fancyhdr}
\title{Consecutive Quadratic Residues And Quadratic Nonresidue Modulo $p$}
\date{}
\author{N. A. Carella}

\begin{document}
\thispagestyle{empty}
\date{}
\maketitle
\textbf{\textit{Abstract}:} Let $p$ be a large prime, and let $k\ll \log p$. A new proof of the existence of any pattern of $k$ consecutive quadratic residues and quadratic nonresidues is introduced in this note. Further, an application to the least quadratic nonresidues $n_p$ modulo $p$ shows that $n_p\ll (\log p)(\log \log p)$.  \let\thefootnote\relax\footnote{ \today \date{} 
}


\section{Introduction}  \label{s8833}
Given a prime $p \geq 2$, a nonzero element $u \in \F_p$ is called a \textit{quadratic residue}, equivalently, a square modulo $p$ whenever the quadratic congruence $x^2-u\equiv 0 \tmod p$ is solvable. Otherwise, it called a \textit{quadratic nonresidue}. A finite field $\F_p$ contains $(p+1)/2$ squares $\mathcal{R}=\{u^2 \tmod p:0\leq u< p/2\}$, including zero. The quadratic residues are uniformly distributed over the interval $[1, p-1]$. Likewise, the quadratic nonresidues are uniformly distributed over the same interval. \\

Let $k \geq 1$ be a small integer. This note is concerned with the longest runs of consecutive quadratic residues and consecutive quadratic nonresidues (or any pattern) 
\begin{equation} \label{eq8833.030}
u, \quad u+1, \quad u+2,\quad \ldots,\quad u+k-1,
\end{equation}
in the finite field $\F_p$, and large subsets $\mathcal{A} \subset \F_p$. Let $N(k,p)$ be a tally of the number of sequences \ref{eq8833.030}.  

\begin{thm}  \label{thm8833.040}  Let \(p\geq 2\) be a large prime, and let $k=O\left(\log p \right )$ be an integer. Then, the finite field $\F_p$ contains $k$ consecutive quadratic residues {\normalfont(}or quadratic nonresidues or any pattern{\normalfont)}. Furthermore, the number of $k$ tuples has the asymptotic formulas
\begin{enumerate}[font=\normalfont, label=(\roman*)]
\item $\displaystyle N(k,p) =   \frac{p}{2^k} \left (1-\frac{1}{p} \right )^k\left (1+O\left (\frac{1}{p}  \right )\right ),$ \tabto{8cm} if $k \geq 1.$\\
    
 \item $\displaystyle N(k,p)  =   \frac{p}{2^k} +O\left (k^2 \right ),$\tabto{8cm} if $k \geq 1.$
 \end{enumerate}
\end{thm} 
The first expression is suitable for applications requiring any integer $k$, and the second expression is suitable for applications requiring small integers $k$. The proofs are assembled in Section \ref{s5599}. The main results are proved using a new counting technique, based on Lemma \ref{lem9933.03}. It provides sharper error terms than the standard technique based on Lemma \ref{lem9933.09}. 
\begin{thm}  \label{thm8833.050}  Let \(p\geq 2\) be a large prime, and let $k=O\left(\log p \right )$ be an integer. Then, for any subset of consecutive elements $\mathcal{A} \subset \F_p$ of cardinality $p^{1-\varepsilon/2}\ll \# \mathcal{A}$ contains $k$ consecutive quadratic residues {\normalfont(}or quadratic nonresidues or any pattern{\normalfont)}, $\varepsilon>0$ is an arbitrary small number. Furthermore, the number of $k$ tuples has the asymptotic formulas
\begin{enumerate}[font=\normalfont, label=(\roman*)]
\item $\displaystyle N(k,p, \mathcal{A}) =   \frac{\#\mathcal{A}}{2^k} \left (1-\frac{1}{p} \right )^k\left (1+O\left (\frac{1}{p}  \right )\right ),$ \tabto{8cm} if $k \geq 1.$\\
    
 \item $\displaystyle N(k,p, \mathcal{A})  =   \frac{\#\mathcal{A}}{2^k} +O\left (k \right ),$\tabto{8cm} if $k \geq 1.$
 \end{enumerate}
\end{thm} 

Quadratic residues $r \in \F_p$ (and quadratic non residues) in finite fields have orders $\ord_p(r)=2$. The analysis and results for $k$ consecutive $d$ power residues or any pattern of $d$ power residues have similar details, but are more complex as the orders of the elements increases. The other result consider an interesting application to the least quadratic nonresidue.\\

The other result considers an interesting application to the least quadratic nonresidue $n_p$ modulo $p$. The current unconditional result for the least quadratic nonresidues in the literature states that 
\begin{equation} \label{eq8833.053}
n_p\ll p^{1/4\sqrt{e}+\varepsilon},
\end{equation}
where $\varepsilon>0$, and the strongest conditional result for primitive character $\chi$ states that 
\begin{equation} \label{eq8833.055}
n_{\chi}(p)\ll (\log p)^{1.37+o(1)},
\end{equation}
see \cite[Corollary 2]{BG13}, and Conjecture \ref{conj1699.000}. The following result is proved here.
\begin{thm}  \label{thm8833.242}  For any large prime \(p\geq 2\), the least quadratic nonresidue is bounded by 
\begin{equation} \label{eq8833.057}
n_p\ll (\log p)(\log  \log p).
\end{equation}
\end{thm} 
The implied constant should be small, perhaps $\leq 20$, see Table \ref{t1159.04} in Section \ref{s1159}. In Section \ref{s1699} several Lemmas are spliced together to prove this result. Section  \ref{s7777} to Section \ref{s5000} cover the supporting materials and other optional topics.\\ 
\section{Quadratic Symbol} \label{s8877}

\begin{dfn} \label{dfn8877.100} {\normalfont Let $p\geq 2$ be a prime, and let $u \in \F_p$. The \textit{quadratic symbol} modulo $p$ of a nonzero element $u$ is defined by 
\begin{equation} \label{eq8877.02}
\left ( \frac{u}{p}\right )=
\begin{cases}
1& \text{ if } u \text{ is a quadratic residue};\\
-1& \text{ if } u \text{ is not a quadratic residue}.\\
\end{cases}
\end{equation} 
}
\end{dfn}
In term of calculations, this can be determined via the Euler criterion 
\begin{equation} \label{eq8877.04}
\left ( \frac{u}{p}\right )=u^{(p-1)/2}\equiv \pm 1 \tmod p.
\end{equation}

\begin{lem}  {\normalfont (Legendre)}\label{lem8877.050}  Let \(p\geq 2\) and $q$ be a pair of distinct primes. Then, the quadratic symbol satisfies the following properties.
\begin{enumerate}[font=\normalfont, label=(\roman*)]
\item $\displaystyle \left ( \frac{a}{p}\right )\equiv a^{(p-1)/2} \tmod p,$ \tabto{8cm} Euler congruence equation.\\
    
 \item $\displaystyle \left ( \frac{ab}{p}\right )  =   \left ( \frac{a}{p}\right )\left ( \frac{b}{p}\right ),$\tabto{8cm} completely multiplicative function.

\item $\displaystyle \left ( \frac{-1}{p}\right )  =   \left ( -1\right )^{(p-1)/2},$\tabto{8cm} evaluation at $a=-1$.

\item $\displaystyle \left ( \frac{2}{p}\right )  =   \left ( -1\right )^{(p^2-1)/8},$\tabto{8cm} evaluation at $a=2$.

\item $\displaystyle \left ( \frac{q}{p}\right )  \left ( \frac{p}{q}\right ) =   \left ( -1\right )^{(p-1)(q-1)/4},$\tabto{8cm} quadratic reciprocity.
 \end{enumerate}
\end{lem}

\section{Some Quadratic Exponential Sums} \label{s7777}
\begin{dfn} \label{dfn7777.100} {\normalfont Let $p\geq 2$ be a prime. The finite Fourier transform of a periodic function $f: \Z \longrightarrow \C$ of period $p$ is defined by
\begin{equation} \label{eq7777.100}
\hat{f}(s)=\frac{1}{\eta_p\sqrt{p}}\sum_{ t\in \F_p} f(t)e^{i 2\pi st/p},		
\end{equation}
where $\eta_p=1$ if $p\equiv 1 \tmod 4$ or $\eta_p=i$ if $p\equiv 3 \tmod 4$.
}
\end{dfn}
Except for a normalizing factor, the standard Gauss sum is a finite Fourier transform of the nonprincipal character $\chi: \Z \longrightarrow \C$, namely,
\begin{equation} \label{eq7777.103}
\tau_s(\chi)=\sum_{ t\in \F_p} \chi(t )e^{i 2\pi ts/p}.
\end{equation} 
\begin{lem} \label{lem7777.102}  The quadratic character mod \(p\) is the unique fix point of the finite Fourier transform. Specifically 
\begin{equation} \label{eq7777.103}
\left ( \frac{s}{p}\right )=\frac{1}{\eta_p\sqrt{p}}\sum_{ t\in \F_p} \left ( \frac{t}{p}\right )e^{i 2\pi st/p}	,
\end{equation} 
where $\eta_p=1$ if $p\equiv 1 \tmod 4$ or $\eta_p=i$ if $p\equiv 3 \tmod 4$.
\end{lem}

\begin{lem} \label{lem7777.09} {\normalfont (Gauss)} If $p\geq 2$ is a prime, then  
\begin{equation} \label{eq7777.07}
\sum_{ u\in \F_p} e^{i 2\pi u^2/p}=
\begin{cases}
\sqrt{p} \text{ if } p \equiv 1 \tmod p;\\
i\sqrt{p} \text{ if } p \equiv 3 \tmod p.\\
\end{cases}
\end{equation} 
\end{lem}
\begin{proof} It is widely available in the literature, exampli gratia, \cite[Theorem 1.1.5]{BW98}, \cite[Section 3.3]{LF00}, \cite[Lemma 3.3]{MT19}.
\end{proof}

\begin{lem} \label{lem7777.12}  Let \(p\geq 2\) be a prime, and let $(x\,| \,p)$ be the quadratic character mod \(p\). If the element \(s\ne 0\), then, 
\begin{equation} \label{eq7777.11}
\sum_{ u\in \F_p} e^{i 2\pi u^2s/p}=\left ( \frac{s^{-1}}{p}\right )\eta_p\sqrt{p},	
\end{equation} 
where $\eta_p=1$ if $p\equiv 1 \tmod 4$ or $\eta_p=i$ if $p\equiv 3 \tmod 4$.
\end{lem}
\begin{proof} Let $\chi(n)=(x\,| \,p) $. Replace the characteristic function of quadratic residue, see Lemma \ref{lem9933.09}, to obtain
\begin{equation} \label{eq7777.15}
\sum_{ u\in \F_p} e^{i 2\pi u^2s/p}=\sum_{ u\in \F_p} \left ( 1+\chi(u)\right ) e^{i 2\pi us/p}=\sum_{ u\in \F_p} \chi(u) e^{i 2\pi us/p}.		
\end{equation} 
The change of variable $z=us$ returns
\begin{equation} \label{eq7777.09}
\sum_{ u\in \F_p} e^{i 2\pi u^2s/p}=\chi(s^{-1})\sum_{ z\in \F_p} \chi(z) e^{i 2\pi z/p}=\chi(s^{-1})\eta_p\sqrt{p}.		
\end{equation}

\end{proof}

\begin{lem} \label{lem7777.25} If $p\geq 2$ is a prime, and $a\ne$ is an integer such that $\gcd(a,p)=1$, then  
\begin{equation} \label{eq7777.27}
\sum_{ u\in \F_p} \left ( \frac{ax^2+bx+c}{p}\right )=
\begin{cases}
-\left ( \frac{a}{p}\right ) & \text{ if } b^2-4ac\not \equiv 0 \tmod p;\\
\left ( \frac{a}{p}\right )(p-1) & \text{ if } b^2-4ac\equiv 0 \tmod p.\\
\end{cases}
\end{equation} 
\end{lem}
\begin{proof} Consult the literature, \cite[Theorem 2.1.2]{BW98}, \cite{LN97}, and similar references.
\end{proof}

\section{Some Incomplete Exponential Sums} \label{s7733}
A classical application of the finite Fourier transform provides nontrivial upper bounds of incomplete character sums. The simplest one uses the quadratic symbol or plain character $\chi$ modulo $q$.

\begin{lem} \label{lem7733.086}  {\normalfont (Polya-Vinogradov)} For \(q\) is a large prime, and a character $\chi\ne1$ modulo $q$,  
\begin{equation} \label{eq7733.106}
\sum_{n \leq x} \chi(n)\ll \sqrt{q} \log q.
\end{equation}
\end{lem}

\begin{proof} Use the finite Fourier transform of $\chi(n)$ as in Lemma \ref{lem7777.102} and the geometric series, and other means. 
\end{proof} 

The distribution, and various properties of the implied constant has a vast literature, and it is a topic of current research, see \cite{BK18}, \cite{FG19}, et alii. Many improved upper bounds for some specific characters such as $\chi(-1)=-1$ or $\chi(-1)=1$ are known. An explicit for the Burgess inequality is stated below.  

\begin{lem} {\normalfont (\cite{TE15})}\label{lem7733.105} Let $p \geq 10^7$ be a prime, and let $\chi$ be character modulo $p$. Let $M$, and $N\geq 1$ be nonnegative integer and let $r\geq1$. Then 
\begin{equation} \label{eq7733.115}
\sum_{M \leq n \leq N+M} \chi(n)\leq2.7 N^{1-1/r}p^{(r+1)/4r^2} (\log p)^{1/r}.
\end{equation}
\end{lem}
At $r=1$ it reduces to the Polya-Vinogradov inequality, and as $r \to \infty$, it becomes a trivial upper bound.

\begin{lem} \label{lem7733.107} Suppose that GRH is true. Then, for any nonprincipal character $\chi$ modulo $q$ and any large number $x$, 
\begin{equation} \label{eq7733.120}
\sum_{n \leq x} \chi(n)\ll \sqrt{q} \log \log q.
\end{equation}
\end{lem}

\begin{proof} This is done in \cite[Theorem 2]{MV77}. 
\end{proof} 

\begin{lem} {\normalfont (Paley)}\label{lem7733.108} There are infinitely many discriminant $q \equiv 1 \tmod q$ for which  
\begin{equation} \label{eq7733.122}
\sum_{n \leq x} \left ( \frac{n}{q}\right )\gg \sqrt{q} \log \log q.
\end{equation}
\end{lem}
\begin{proof} The original version appears in \cite{PR32} and a recent version is given in \cite[Theorem 9.24]{MV07}. 
\end{proof}

\section{Characteristic Functions For Quadratic Residues} \label{9933}
The standard characteristic function of quadratic residues and quadratic nonresidues are induced by the quadratic symbol. 
\begin{lem} \label{lem9933.09}  Let \(p\geq 2\) be a prime, and let $(x\,| \,p)$ be the quadratic character mod \(p\). If \(u\in\mathbb{F}_p\) is a nonzero element, then, 
\begin{enumerate}[font=\normalfont, label=(\roman*)]
\item $\displaystyle \Psi_2 \left(u\right)=\frac{1}{2} \left(1+ \left ( \frac{u}{p}\right)\right) =
\left \{\begin{array}{ll}
1 & \text{ if } u^{(p-1)/2}\equiv 1 \tmod p,  \\
0 & \text{ if } u^{(p-1)/2}\equiv -1 \tmod p. \\
\end{array} \right .$\\
    
 \item $\displaystyle \overline{\Psi}_2(u) =\frac{1}{2} \left(1- \left ( \frac{u}{p}\right)\right)=
\left \{\begin{array}{ll}
1 & \text{ if } u^{(p-1)/2}\equiv -1 \tmod p,  \\
0 & \text{ if } u^{(p-1)/2}\equiv 1 \tmod p, \\
\end{array} \right . $
 \end{enumerate}
are the characteristic functions for quadratic residues and quadratic non residues modulo $p$ respectively in the finite field $\F_p$.
\end{lem} 

A new representation of the characteristic function for quadratic residues and quadratic nonresidues are introduced below.
\begin{lem} \label{lem9933.03}
Let \(p\geq 2\) be a prime, and let \(\tau\) be a primitive root mod \(p\). If \(u\in\mathbb{F}_p\) is a nonzero element, then,

\begin{enumerate}[font=\normalfont, label=(\roman*)]
\item $\displaystyle \Psi_2 (u)=\sum _{0\leq n<(p-1)/2} \frac{1}{p}\sum _{0\leq m\leq p-1} e^{i2\pi \left (\tau ^{2n}-u\right)m/p}
=\left \{
\begin{array}{ll}
1 & \text{ if }  u^{(p-1)/2}\equiv 1 \tmod p,  \\
0 & \text{ if }  u^{(p-1)/2}\equiv -1 \tmod p. \\
\end{array} \right .$\\
    
 \item $\displaystyle \overline{\Psi}_2(u) =\sum _{0\leq n<(p-1)/2} \frac{1}{p}\sum _{0\leq m\leq p-1} e^{i2\pi \left (\tau ^{2n+1}-u\right)m/p}
=\left \{
\begin{array}{ll}
1 & \text{ if }  u^{(p-1)/2}\equiv -1 \tmod p,  \\
0 & \text{ if }  u^{(p-1)/2}\equiv 1 \tmod p. \\
\end{array} \right .$
 \end{enumerate}
\end{lem}
\begin{proof}(i) For a fixed $u\ne 0$, the finite field $\F_p$ equation 
\begin{equation} \label{eq9933.030}
\tau ^{2n}-u=0
\end{equation}
has a unique solution $n=n_0 \in \{0,1,2, \ldots (p-1)/2-1\}$ if and only if $0\ne u=\tau ^{2n_0}$ is a quadratic residue modulo $p$. This, in turns, implies 
that the inner exponential sum collapses to $p$. Specifically,
\begin{equation} \label{eq9933.033}
\sum _{0\leq m\leq p-1} e^{i2\pi \left (\tau ^{2n}-u\right)m/p}
=\left \{
\begin{array}{ll}
p & \text{ if }  u^{(p-1)/2}\equiv 1 \tmod p,  \\
0 & \text{ if }  u^{(p-1)/2}\equiv -1 \tmod p. \\
\end{array} \right .
\end{equation}
Otherwise, $\tau ^{2n}-u\ne 0$, which implies that the inner exponential sum vanishes.
\end{proof}

\section{Estimate For The Sum $T(k,p)$} \label{s1299}
The calculations for the sum 
\begin{equation} \label{eq1299.00}
T(k,p) =\sum_{ 0\ne n\in \F_p}   \prod_{0 \leq i\leq k-1}\left (\frac{1}{p}\sum_{0\leq n_i< (p-1)/2} 1 \right ) 
\end{equation} 
assumes the existence of a sequence of $k$ consecutive quadratic residues 
\begin{equation} \label{eq1299.030}
u, \quad u+1, \quad u+2,\quad \ldots,\quad u+k-1
\end{equation}
in the finite field $\F_p$. However, it is valid for any pattern of quadratic residues and nonresidues.

\begin{lem} \label{lem1299.06}  Let \(p\geq 2\) be a large prime, let $k =O\left(\log p\right )$ be an integer, then, 
\begin{enumerate}[font=\normalfont, label=(\roman*)]
\item $\displaystyle T(k,p) =   \frac{p-1}{2^k} \left (1-\frac{1}{p} \right )^k,$ \tabto{8cm} if $k \geq 1.$\\
    
 \item $\displaystyle T(k,p)  =   \frac{p}{2^k} +O\left (k \right ),$\tabto{8cm} if $k \geq 1.$ 
 \end{enumerate}
\end{lem}

\begin{proof} (i) Routine calculations return
\begin{eqnarray} \label{eq5599.020}
T(k,p)&=&\sum_{ 0\ne n\in \F_p}   \prod_{0 \leq i\leq k-1}\left (\frac{1}{p}\sum_{0\leq n_i< (p-1)/2} 1 \right )  \\
&=&\sum_{ 0\ne n\in \F_p}   \left (\frac{1}{p} \cdot \frac{p-1}{2} \right )^k  \nonumber\\
&=&\frac{p-1}{2^k} \left (1-\frac{1}{p} \right )^k  \nonumber.
\end{eqnarray}
(ii) For small integer $k =O\left(\log p\right )$, the binomial series leads to
\begin{eqnarray} \label{eq5599.020}
T(k,p)
&=&\frac{p-1}{2^k} \left (1-\frac{1}{p} \right )^k \\
&=&\frac{p-1}{2^k} \left (1-\binom{k}{1}\frac{1}{p}+ \binom{k}{2}\frac{1}{p^2}+\cdots +\frac{(-1)^k}{p^k}\right )
\nonumber\\
&\leq&\frac{p-1}{2^k} \left (1+k\binom{k}{k/2}\frac{1}{p}\right )
\nonumber\\
&=&\frac{p}{2^k} +O\left ( k \right )
\nonumber,
\end{eqnarray}
since the central binomial coefficient $\binom{k}{k/2}\leq 2^k$.
\end{proof}
\section{The Estimates For The Sum $U(k,p)$}  \label{s1499}
The exponential sums over finite fields $\mathbb{F}_p$ studied in Section \ref{s7777} are used to estimate the sum 
\begin{equation} \label{eq1499.00}
U(k,p) =	\sum_{ 0\ne u\in \F_p}\prod_{0 \leq r\leq k-1} \left ( \frac{1}{p}
	\sum_{\substack{0 \leq n_r<(p-1)/2 \\ 0<m_r \leq p-1}} e^{i2 \pi \left((\tau ^{2n_r}-u-a_r)m_r\right)}\right ) ,      
\end{equation}
which is used to prove Theorem \ref{thm8833.040}.

\begin{lem} \label{lem1499.06}  Let \(p\geq 2\) be a large prime, let \(k\geq 1\) be an integer, and let \(\tau\) be a primitive root mod \(p\). If the elements \(u+a_r=v_r^2\ne0\) are quadratic residues for $r=0,1,2,...,k-1$, then, 
\begin{enumerate}[font=\normalfont, label=(\roman*)]
\item $\displaystyle U(k,p) =O\left (\frac{1}{2^k} \left (1+\frac{1}{p} \right )^k\right ),$ \tabto{8cm} if $k \geq 1.$\\
    
 \item $\displaystyle U(k,p) =O\left (\frac{k}{2^k} \right ),$\tabto{8cm} if $k \geq 1.$
 \end{enumerate}
\end{lem}

\begin{proof}  (i) Rewrite the multiple finite sum \eqref{eq1499.00} as a product
\begin{eqnarray} \label{eq1499.02}
U(k,p) &=&	\sum_{ 0\ne u\in \F_p}  \left ( \frac{1}{p}
	\sum_{\substack{0 \leq n_0<(p-1)/2 \\ 0<m_0 \leq p-1}} e^{i2 \pi \left((\tau ^{2n_0}-u-a_0)m_0\right)}\right )   \\
&& \qquad \qquad 	\times \prod_{1 \leq r\leq k-1} \left ( \frac{1}{p}
	\sum_{\substack{0 \leq n_r<(p-1)/2 \\ 0<m_r \leq p-1}} e^{i2 \pi \left((\tau ^{2n_r}-u-a_r)m_r\right)}\right ) \nonumber\\
&=& 	\sum_{ 0\ne u\in \F_p}U_1   	\times U_2 \nonumber.      
\end{eqnarray}

Merging Lemma \ref{lem1499.09} and Lemma \ref{lem1499.79}, (using Holder inequality is optional), return
\begin{eqnarray} \label{eq1499.30}
U(k,p) &=  & 	\sum_{ 0\ne u\in \F_p}U_1 \cdot U_2  \\
&=  & 	\sum_{ 0\ne u\in \F_p}\frac{e^{-i 2 \pi a_0/p}}{4p}\left (1+\eta_p^2 p\left ( \frac{u^{-1}}{p}\right)\right ) \cdot \left (\frac{p+1}{2p} \right )^{k-1} \nonumber\\
&= &   \frac{e^{-i 2 \pi a_0/p}}{4p}\left (\frac{p+1}{2p} \right )^{k-1}\sum_{ 0\ne u\in \F_p}1+\frac{\eta_p^2e^{-i 2 \pi a_0/p}}{4}\left (\frac{p+1}{2p} \right )^{k-1}\sum_{ 0\ne u\in \F_p}\left (\frac{u^{-1}}{p}\right)\nonumber\\
&=& \frac{e^{-i 2 \pi a_0/p}}{4p}\left (\frac{p+1}{2p} \right )^{k-1}(p-1)\nonumber,
\end{eqnarray}
since
\begin{equation}\label{eq1499.33}
\sum_{ 0\ne u\in \F_p}\left (\frac{u^{-1}}{p}\right)=0.    
\end{equation}
The absolute value $\left |U(k,p) \right | $ of the last expression in \eqref{eq1499.30} can be rewritten as in the Lemma. (ii) Same as the proof in Lemma \ref{lem1299.06}-ii.
\end{proof}

\begin{lem} \label{lem1499.09}  Let \(p\geq 2\) be a large prime, and let \(\tau\) be a primitive root mod \(p\). If the element \(u+a_0=v_0^2\ne0\) is a quadratic residue, then, 
	\begin{equation} \label{eq1499.03}
U_1=	\frac{1}{p}
\sum_{\substack{0\leq n<(p-1)/2\\0<m \leq p-1}}e^{i2 \pi \left((\tau ^{2n}-u-a_0)m\right)}    =\frac{e^{-i 2 \pi a_0/p}}{4p}\left (1+\eta_p^2 p\left ( \frac{u^{-1}}{p}\right) \right )  , 
	\end{equation} 
where $|\eta_p|=\pm1$.
\end{lem}

\begin{proof} Rearrange the finite sum as 
\begin{eqnarray} \label{eq1499.05}
U_1&=&\frac{1}{p}
	\sum_{\substack{0\leq n<(p-1)/2\\0<m \leq p-1}}e^{i2 \pi \left((\tau ^{2n}-u-a_0)m\right)}  \nonumber\\
&= & \frac{ e^{-i 2 \pi a_0/p}}{p}\sum_{ 0<m\leq p-1} e^{-i 2 \pi um/p}  \sum_{0\leq n <(p-1)/2} e^{i 2 \pi m\tau ^{2n}/p}  \nonumber\\
&= & \frac{e^{-i 2 \pi a_0/p}}{2p} \sum_{ 0<m\leq p-1} e^{-i 2 \pi um/p} \sum_{0\leq n < p-1} e^{i 2 \pi m\tau ^{2n}/p} .
\end{eqnarray} 
Let $s=\tau^{n}$. An application of Lemma \ref{lem7777.09}, yields
\begin{equation} \label{eq1499.04}
1+2\sum_{0\leq n < p-1} e^{i 2 \pi m\tau ^{2n}/p}=\sum_{0\leq n < p-1} e^{i 2 \pi ms^{2}/p}   
= \left (\frac{m^{-1}}{p}\right )\eta_p\sqrt{p}.
\end{equation} 
Substituting this into the previous equation returns
\begin{eqnarray} \label{eq1499.07}
U_1&=& \frac{e^{-i 2 \pi a_0/p}}{4p}  \sum_{ 0<m\leq p-1} e^{-i 2 \pi um/p}  \left (-1+ \left ( \frac{m^{-1}}{p}\right )\eta_p\sqrt{p} \right ) \\
 &= &\frac{e^{-i 2 \pi a_0/p}}{4p} \left (1+\eta_p\sqrt{p} \sum_{ 0<m\leq p-1} \left ( \frac{m^{-1}}{p}\right ) e^{-i 2 \pi um/p}  \right)  \nonumber\\
 &= &\frac{e^{-i 2 \pi a_0/p}}{4p}  \left (1+\eta_p^2 p\left ( \frac{u^{-1}}{p}\right )   \right)  \nonumber,
\end{eqnarray}
where $|\eta_p|=\pm 1$. 
\end{proof}

\begin{lem} \label{lem1499.79}  Let \(p\geq 2\) be a large prime, and let \(\tau\) be a primitive root mod \(p\). If the elements \(u+a_r= v_r^2\ne0\) are quadratic nonresidues, then, 
	\begin{equation} \label{eq1499.71}
U_2=	\prod_{1 \leq r\leq k-1} \left ( \frac{1}{p}
	\sum_{\substack{0 \leq n_r<(p-1)/2 \\ 0<m_r \leq p-1}} e^{i2 \pi \left((\tau ^{2n_r}-u-a_r)m_r\right)}\right )  = \left (\frac{p+1}{2p} \right )^{k-1}.  
	\end{equation} 
\end{lem}

\begin{proof}The hypothesis $u+a_r= v_r^2$ for $r=1,2, \ldots,k-1$, is used to determine the value of each incomplete exponential sum. Start with a complete exponential sum and break it up into two subsums:
\begin{eqnarray} \label{eq1499.73}
1&=&\frac{1}{p} \sum_{\substack{0 \leq n_r<(p-1)/2 \\ 0\leq m_r \leq p-1}} e^{i2 \pi \left((\tau ^{2n_r}-u-a_r)m_r\right)}    \\
&= &  \frac{1}{p} \sum_{0 \leq n_r<(p-1)/2 }1\quad + \quad \frac{1}{p} \sum_{\substack{0 \leq n_r<(p-1)/2 \\ 0<  m_r \leq p-1}} e^{i2 \pi \left((\tau ^{2n_r}-u-a_r)m_r\right)}\nonumber\\
&= &\frac{p-1}{2p} \quad + \quad \frac{1}{p} \sum_{\substack{0 \leq n_r<(p-1)/2 \\ 0<  m_r \leq p-1}} e^{i2 \pi \left((\tau ^{2n_r}-u-a_r)m_r\right)}\nonumber.
\end{eqnarray}
Solving for the incomplete exponential sum on the right side of \eqref{eq1499.73} yields 
\begin{equation} \label{eq1499.77}
 \frac{p+1}{2p}=1-\frac{p-1}{2p},  
	\end{equation} 
which is basically the trivial value. Taking the product of all the incomplete exponential sums yields
\begin{eqnarray} \label{eq1499.75}
U_2 &= &  \left (\frac{1}{p} \sum_{\substack{0 \leq n_1<(p-1)/2 \\ 0<m_1 \leq p-1}} e^{i2 \pi \left((\tau ^{2n_1}-u-a_1)m_1\right)}  \right ) \nonumber\\
&& \qquad \qquad \times \cdots \times 
\left ( \frac{1}{p}\sum_{\substack{0 \leq n_1<(p-1)/2 \\ 0<m_1 \leq p-1}} e^{i2 \pi \left((\tau ^{2n_{k-1}}-u-a_{k-1})m_{k-1}\right)} \right )  \nonumber \\
&= &  \frac{p+1}{2p}  
\times \cdots \times \frac{p+1}{2p} \nonumber\\
&= &\left (\frac{p+1}{2p} \right )^{k-1}.
\end{eqnarray}
\end{proof}

\section{Consecutive Quadratic Residues And Nonresidues} \label{s5599}
Consecutive quadratic nonresidues is one of the simplest configuration of a subset of two or more consecutive quadratic nonresidues. The earliest attempts are surveyed in \cite{GR04}, \cite{HP06}, et alii. A more general result was proved by Carlitz \cite[Theorem 3]{CL56} using a counting technique based on Lemma \ref{lem9933.09}. More precisely, the number of $k$ consecutive quadratic residue symbols or any pattern of quadratic residue and quadratic nonresidue symbols
\begin{equation} \label{eq8833.005}
 \left ( \frac{u}{p}\right)=\varepsilon_0, \quad  \left ( \frac{u+1}{p}\right)=\epsilon_1,\quad\cdots, \quad \left ( \frac{u+k-1}{p}\right)=\epsilon_{k-1},
\end{equation}
where $\epsilon_n=\pm1$, in the finite field $\F_p$ has the asymptotic formula
\begin{eqnarray} \label{eq8833.020}
N(k,p)&=&\frac{1}{2^{k}}\sum_{0\leq u\leq p-1}  \left(1+ \left ( \frac{u}{p}\right)\epsilon_0\right)\left(1+ \left ( \frac{u+1}{p}\right)\epsilon_1\right)\cdots \left(1+ \left ( \frac{u+k-1}{p}\right)\epsilon_{k-1}\right)\nonumber\\
&=&\frac{p}{2^{k}}+E(k,p),
\end{eqnarray}
see Lemma \ref{lem9933.09} for details on the the characteristic functions. An explicit error term $E(k,p)=\pm(k+1)(3+\sqrt{p})$ is proved in 
\cite[Corollary 5]{PR92}. A slightly different proof appears in a new survey \cite[Theorem 5.6]{MT19}. For $k\leq 2$, the exponential sums involved 
have exact evaluations, and there are no error terms. But, in general, for $k\geq 3$, with very few exceptions, the exponential sums are estimated, and have 
error terms of the forms $E(k,p)=O(k\sqrt{p})$. A new and sharper proof and counting technique based on Lemma \ref{lem9933.03} is given here.\\

Let  $ a_0, a_1, a_2, \ldots,a_{k-1}$ be a sequence of distinct and increasing integers. Let $p \geq 2$ be a large prime, and let $\tau \in \F_p$ be a primitive root. A pattern of $k$ consecutive quadratic residues and quadratic nonresidues $u+a_0, u+a_1, u+a_2, \ldots, u+a_{k-1}$ exists if and only if the system of equations 
\begin{equation} \label{eq5599.00}
\tau^{2n_0}=u+a_0, \quad \tau^{2n_1}=u+a_1, \quad\tau^{2n_2}=u+a_2, \quad \ldots, \quad \tau^{2n_{k-1}}=u+a_{k-1},
\end{equation}
has one or more solutions. A solution consists of a $k$-tuple $n_0,n_1,\ldots, n_{k-1}$ of integers such that $0\leq n_i<(p-1)/2$ for $i=0,1, \ldots, k-1$, and some $u \in \F_p$. Let 

\begin{equation} \label{eq5599.02}
N(k,p)=\#\left \{ u \in \F_p: \ord_p (u+a_i)=2 \right \}
\end{equation}
for $i=0, 1, \ldots, k-1$, denotes the number of solutions. 

\begin{proof} (Theorem \ref{thm8833.040}): The total number of solutions is written in terms of characteristic function for quadratic residues, see Lemma \ref{lem9933.03}, as
\begin{eqnarray} \label{eq5599.018}
N(k,p)&=&\sum_{0 \ne u\in \F_p}  \Psi_2 \left(u+a_0\right)\Psi_2 \left(u+a_1\right)\cdots \Psi_2 \left(u+a_{k-1}\right) \\
&=&\sum_{ 0\ne u\in \F_p}  \prod_{0 \leq i\leq k-1} \left (\frac{1}{p}\sum_{\substack{0\leq n_i\leq (p-1)/2\\
0\leq m_i\leq p-1}}\psi \left((\tau ^{2n_i}-u-a_i)m_i)\right) \right )  \nonumber\\
&=&T(k,p)\quad  +\quad U(k,p)\nonumber.
\end{eqnarray} 
The term $T(k,p)$, which is determined by the indices $m_0=m_1=\cdots=m_{k-1}=0$, has the form 
\begin{equation} \label{eq5599.020}
T(k,p)=\sum_{ 0\ne u\in \F_p}   \prod_{0 \leq i\leq k-1}\left (\frac{1}{p}\sum_{0\leq n_i< (p-1)/2} 1 \right )  ,
\end{equation} 
and the term $U(k,p)$, which is determined by the indices $m_0\ne0,m_1\ne0,\ldots,m_{k-1}\ne0$, has the form 
\begin{equation} \label{eq5599.022}
U(k,p)=\sum_{ 0\ne u\in \F_p}  \prod_{0 \leq i\leq k-1} \left (\frac{1}{p}\sum_{\substack{0\leq n_i\leq (p-1)/2\\
1\leq m_i\leq p-1}} \psi \left((\tau ^{2n_i}-u-a_i)m_i)\right) \right )  .
\end{equation}
(i) Applying Lemma \ref{lem1299.06}-i to the term $T(k,p)$, and Lemma \ref{lem1499.06}-i to the term $U(k,p)$, yield
\begin{eqnarray} \label{eq5599.038}
N(k,p)
&=&T(k,p) \quad+\quad U(k,p)\\
&=&  \frac{p-1}{2^k} \left (1-\frac{1}{p} \right )^k+O\left (\frac{1}{2^k} \left (1+\frac{1}{p} \right )^k\right )\nonumber\\
&=&  \frac{p}{2^k} \left (1-\frac{1}{p} \right )^k\;\left (1+O\left (\frac{1}{p}  \right )\right )\nonumber\\
&>&0 \nonumber,
\end{eqnarray} 
for all sufficiently large primes $p \geq 2$. 

(ii) Applying Lemma \ref{lem1299.06}-ii to the term $T(k,p)$, and Lemma \ref{lem1499.06} to the term $U(k,p)$, yield
\begin{eqnarray} \label{eq5599.138}
N(k,p)
&=&T(k,p) \quad+\quad U(k,p)\\
&=&  \frac{p}{2^k} +O\left (k \right )+O\left (\frac{k}{2^k} \right )\nonumber\\
&=&  \frac{p}{2^k} +O\left (k \right )\nonumber\\
&>&0 \nonumber,
\end{eqnarray} 
for all sufficiently large primes $p \geq 2$. 
\end{proof}

\section{Synopsis Of Upper Bounds For Quadratic Nonresidues}
The mathematical literature has many estimates for the least quadratic nonresidue $n_p$ modulo $p$. A list of the most frequently encountered upper bounds is complied below.
\begin{enumerate} [font=\normalfont, label=(\arabic*)]
\item  $\displaystyle n_p \leq p^{1/2}+1$, \tabto{4cm}derived using an ad hoc elementary argument, see \cite[Theorem 3.9 ]{NZ91}, \cite{SH83}, etc. 
\item  $\displaystyle n_p \ll p^{1/2}\log p$, \tabto{4cm}derived from the Polya-Vinogradov inequality.
\item  $\displaystyle n_p \leq p^{1/2\sqrt{e}+\varepsilon}$, \tabto{4cm}derived from the Polya-Vinogradov inequality for any $\varepsilon>0$.
\item  $\displaystyle n_p \leq p^{1/4\sqrt{e}+\varepsilon}$, \tabto{4cm}derived from the Burgess inequality for any $\varepsilon>0$,  see Lemma \ref{lem7733.105}, and \cite{BD63}.
\item  $\displaystyle n_p \ll p^{\varepsilon}$, \tabto{4cm}the Vinogradov conjecture, where $\varepsilon>0$, see \cite{LY42}, and the literature.
\item  $\displaystyle n_p \leq 2(\log p)^{2}$, \tabto{4cm}derived from the GRH, see \cite{AN52}, \cite{BE85}.
\end{enumerate}

\section{The Least Quadratic Nonresidue}\label{s1699}
Two slightly different heuristics for the conjectured upper bound of the least quadratic nonresidue are given in \cite{PP12} and \cite{TE15}. These are summarized below.

\begin{conj}\label{conj1699.000} For every large prime $p\geq 3$, 
\begin{equation} \label{eq1699.001}
 n_p\ll (\log p)(\log  \log p).
\end{equation}
\end{conj}
The heuristic is based on the proportion of primes $p$ such that the $n$th prime $p_n$ is the least quadratic nonresidue modulo $p$. The proportion has a geometric distribution, and its proof is based on quadratic reciprocity and Dirichlet theorem for primes in arithmetic progressions. For example, the probability for each $n\geq 1$ is given by limit
\begin{equation} \label{eq1699.030}
 P(n_p=p_n)=\lim_{x\to \infty} \frac{\#\{p\leq x: n_p=p_n\}}{\pi(x)}=\frac{1}{2^n}.
\end{equation}
Note that the form of the main term in Theorem \ref{thm8833.040}-ii implies that quadratic residues (or quadratic non residues) in a finite field $\F_p$ are independent or nearly independent random variables $X=X(p)$ with probability 
\begin{equation} \label{eq1699.032}
P\left ( \ord_p\left (X\right )=2\right)=\frac{1}{2}+O \left ( \frac{1}{p^{\varepsilon}}\right),
\end{equation}
where $\varepsilon>0$ is an arbitrary small number.\\

On average, the expected value $\overline{n}$ of the least quadratic residue $n_p$ is quite small
\begin{equation} \label{eq1699.002}
 \overline{n}=\lim_{x\to \infty} \frac{1}{\pi(x)}\sum_{2<p \leq x}n_p=\sum_{n \geq 1}\frac{p_n}{2^n}=3.67464 \ldots,
\end{equation}
where $\pi(x)=\{p \leq x\}$, and $p_n$ is the $n$th prime, see \cite[p.\ 253]{BS96}, and \cite{MP13} for the generalized concept.

\begin{proof} (Theorem  \ref{thm8833.242}) To obtain a reductio ad absurdum, suppose that there exists a prime $p \geq 2$ for which $n_p> (\log p)(\log \log p)$. Let $k=(\log p)(\log \log  p)$. This implies that the finite field $\F_p$ contains a sequence of $k$ consecutive quadratic residues
\begin{equation} \label{eq1699.030}
u, \quad u+1, \quad u+2,\quad \ldots,\quad u+k-1.
\end{equation}
This immediately implies that
\begin{equation} \label{eq1699.005}
 \left ( \frac{u}{p}\right)=1, \quad  \left ( \frac{u+1}{p}\right)=1,\quad\cdots, \quad \left ( \frac{u+k-1}{p}\right)=1.
\end{equation}
By Theorem \ref{thm8833.040}, the total number of such sequences of quadratic residues of length $k$ is
\begin{equation} \label{eq1699.006}
N(k,p)= \frac{p}{2^k} \left (1-\frac{1}{p} \right )^k\left (1+O\left (\frac{1}{p}  \right )\right )\gg 1.
\end{equation}
Taking logarithm, and simplifying return
\begin{eqnarray} \label{eq5599.14}
\log p-k \log 2+k\log\left (1-\frac{1}{p} \right ) +\log\left (1+O\left (\frac{1}{p}  \right )\right )  \gg 0.
\end{eqnarray}

Rearranging it, and replacing $k=(\log p)(\log \log  p)$ give
 \begin{eqnarray} \label{eq5599.14}
\log p
&\gg&k \log 2+O\left (\frac{ k}{p} \right ) -\log\left (1+O\left (\frac{1}{p}  \right )\right )\\
&\gg&(\log p)(\log  \log p) \log 2+O\left (\frac{ (\log p)(\log  \log p)}{p} \right ) +\log C_0\nonumber,
\end{eqnarray}
where $C_0\approx 1$. Clearly, this is false.  Hence, a finite field $\F_p$ contains a quadratic nonresidues $n_p\ll(\log p)(\log  \log p)$. 
\end{proof}

The best upper bound for the parameter $H\geq 0$ for which  a character $\chi(n)$ modulo $p$ is constant on the interval $[N,N+H]$ is $H<7.07p^{1/4}\log p$ for large primes, see \cite[Theorem 1.1]{MK10}, and \cite{TE12}. The above result in Theorem \ref{thm8833.242} provides an improved and effective upper bound $H\ll (\log p)(\log \log p)$ for the parameter $H\geq 0$. While the result $N(k,p)=p/2^k+O(k\sqrt{p})$ 
in \cite[Corollary 5]{PR92} provides an effective lower bound $H\gg \log p$.
\begin{cor} Let $p\geq 2$ be a large prime, let $\chi$ be a nonprincipal character modulo $p$, and let $x \geq 1$ be a large real number. Define the real value function
\begin{equation} \label{eq1699.505}
 f(x)=\sum_{0 \leq n\leq x} \chi(n).
\end{equation}
Then, $f:\R \longrightarrow \Z$ satisfies the following properties.
\begin{enumerate} [font=\normalfont, label=(\roman*)]
\item  $\displaystyle f(x)=f(x+p) $, \tabto{7cm}is periodic of period $p$.\\

\item  $\displaystyle f(x) \ll p^{1/2}\log p$, \tabto{7cm}is of absolute bounded variation on the real line.\\

\item  $\displaystyle f(N)<f(N+1)< \cdots < f(N+H)$,  \tabto{7cm}is monotonically increasing on a short interval $[N, N+H]$ if and only if 
$H\ll (\log p)(\log \log p)$, for any $N\geq0$.\\

\item  $\displaystyle f(N)>f(N+1)> \cdots > f(N+H)$,  \tabto{7cm}is monotonically decreasing on a short interval $[N, N+H]$ if and only if 
$H\ll (\log p)(\log \log p)$, for any $N\geq0$.\\
\end{enumerate}
\end{cor}

\section{Recursive Algorithm}\label{S1234}
Let $p$ be a prime, let $n_p$ denotes the least quadratic nonresidue modulo $p$. The Burgess upper bound of the least quadratic nonresidue claims that
\begin{equation} \label{eq1234.000}
n_p\leq c_0p^{\frac{1}{4\sqrt{e}}+\varepsilon},  
\end{equation} 
where $c_0>0$, is a constant, and $\varepsilon>0$ is a small number, see \cite{BG13} for a survey and discussion. A recursive technique based on the Vinogradov trick for generating sharper upper bounds is introduced here. It will be demonstrated that a few iterations of the algorithm leads to the new upper bound
\begin{equation} \label{eq1234.002}
n_p\leq c_3p^{\frac{1}{4e\sqrt{e}}+6\varepsilon},  
\end{equation}
where $c_3>0$ is a constant. The numerical values of the exponents of the upper bounds \eqref{eq1234.000} and \eqref{eq1234.002} are 
\begin{multicols}{2}
\begin{enumerate}[font=\normalfont, label=(\arabic*)]
\item $\displaystyle \frac{1}{4\sqrt{e}}=0.151632664928158\ldots,$\\
\item $\displaystyle \frac{1}{4e\sqrt{e}}=0.0557825400371075\ldots.$
\end{enumerate}
\end{multicols}

For sufficiently large prime $p$, and very small $\varepsilon>0$, the last $2$ iterations amounts to a power saving by a factor of 
\begin{equation} \label{eq1234.229}
p^{\frac{1}{4\sqrt{e}}+\varepsilon-\frac{1}{4e\sqrt{e}}-6\varepsilon}=p^{0.0958501248910509-5\varepsilon}. 
\end{equation}

\begin{thm} \label{thm1234.000}{\normalfont (Vinogradov trick)} Let $p$ be a prime, let $x<p$, and let $\chi\ne 1$ be the quadratic character modulo $p$. If $\varepsilon>0$ is a small number, and  
\begin{equation} \label{eq1234.020}
\sum_{n\leq x}\chi(n)=o(x),
\end{equation}
then there exists $n\leq x^{\frac{1}{\sqrt{e}}+\varepsilon}$ such that $\chi(n)=-1$.
\end{thm}
A recent proof appears in \cite[Theorem 2.4]{MT19}, and the earliest proof in \cite{VI18}. The generalization to arbitrary characters $\chi\ne1$ modulo $p$, and different approaches to the proofs are also available in the literature. Here, this result is turned into a recursive algorithm.

\begin{thm} \label{thm1234.005} Let $p$ be a prime, let $n_p$ denotes the least quadratic nonresidue modulo $p$. If $\varepsilon>0$ is a small number, and let $\chi\ne 1$ be the quadratic character modulo $p$, then
\begin{enumerate}[font=\normalfont, label=(\roman*)]
\item $\displaystyle \sum_{n\leq p}\chi(n)\leq c_2p^{1/4e\sqrt{e}+6\varepsilon},$\\
\item $\displaystyle n_p\leq c_2p^{\frac{1}{4e\sqrt{e}}+6\varepsilon},$
\end{enumerate}
where $c_2>0$ is a constant.
\end{thm}
\begin{proof} The upper bound is generated by a few iterations of the "recursive" Vinogradov trick given below. \\

\textbf{Iteration 1.} Let $X_0=c_0p^{1/4+2\varepsilon}$. The standard Burgess exponential sum inequality is 
\begin{equation} \label{eq1234.010}
\sum_{n\leq X_0}\chi(n)\leq c_0p^{\frac{1}{4}+\varepsilon}=o(X_0),
\end{equation}
where $c_0>0$ is a constant, see \cite{BG13}, \cite{MT19}, et alii, detailed discussions. The first iteration of the Vinogradov trick yields
\begin{eqnarray} \label{eq1234.015}
n_p&\leq &X_0^{\frac{1}{\sqrt{e}}+\varepsilon}\\
&\leq &\left (c_0p^{\frac{1}{4}+2\varepsilon}\right )^{\frac{1}{\sqrt{e}}+\varepsilon}\nonumber \\
&\leq &c_1p^{\frac{1}{4\sqrt{e}}+\varepsilon_1}\nonumber,
\end{eqnarray}
where $2(1/\sqrt{e}+\varepsilon)\varepsilon+\varepsilon/4= \varepsilon_1\leq 3\varepsilon$, and $c_1>0$ is a constant. This is the standard unconditional upper bound for quadratic nonresidues in \eqref{eq1234.000}. This iteration is well known in the number theory literature.\\

\textbf{Iteration 2.} Let $X_1=c_1p^{1/4\sqrt{e}+4\varepsilon}$. The corresponding exponential sum inequality is
\begin{equation} \label{eq1234.110}
\sum_{n\leq X_1}\chi(n)\leq c_1p^{\frac{1}{4\sqrt{e}}+3\varepsilon}=o(X_1).
\end{equation}
The second iteration of the Vinogradov trick yields
\begin{eqnarray} \label{eq1234.115}
n_p&\leq &X_1^{\frac{1}{\sqrt{e}}+\varepsilon}\\
&\leq &\left (c_1p^{\frac{1}{4\sqrt{e}}+4\varepsilon}\right )^{\frac{1}{\sqrt{e}}+\varepsilon}\nonumber \\
&\leq &c_2p^{\frac{1}{4e}+\varepsilon_2}\nonumber,
\end{eqnarray}
where $4(1/\sqrt{e}+\varepsilon)\varepsilon+\varepsilon/4\sqrt{e}= \varepsilon_2\leq 4\varepsilon$, and $c_2>0$ is a constant.\\

\textbf{Iteration 3.} Let $X_2=c_2p^{1/4e+5\varepsilon}$. The corresponding exponential sum inequality is
\begin{equation} \label{eq1234.220}
\sum_{n\leq X_2}\chi(n)\leq c_2p^{\frac{1}{4e}+4\varepsilon}=o(X_2).
\end{equation}
The third iteration of the Vinogradov trick yields
\begin{eqnarray} \label{eq1234.225}
n_p&\leq &X_2^{\frac{1}{\sqrt{e}}+\varepsilon}\\
&\leq &\left (c_2p^{\frac{1}{4e}+5\varepsilon}\right )^{\frac{1}{\sqrt{e}}+\varepsilon}\nonumber \\
&\leq &c_3p^{\frac{1}{4e\sqrt{e}}+\varepsilon_3}\nonumber,
\end{eqnarray}
where $5(1/\sqrt{e}+\varepsilon)\varepsilon+\varepsilon/4e= \varepsilon_3\leq 6\varepsilon$, and $c_3>0$ is a constant.
\end{proof}

\section{Experimental Data} \label{s1159}
The numerical data \cite[Table 1]{MT19}, an expanded version is duplicated below, suggests that the constant is $ c_p=n_p/(\log p)(\log  \log p)\leq 20 $ for all primes $ p\geq 2 $.
\begin{table}[h!]
\centering
\caption{Numerical Data for the Least Quadratic Nonresidue Modulo $p$.} \label{t1159.04}
\begin{tabular}{l|l|l| r|r}
$n$&$n_p=p_n$&$p$&$(\log p)(\log \log p)$&$c_p$\\
\hline
1&$2$&   $3$   &$0.10$&$20.00  $\\
2&$3$&  $7$   &$1.30$&$2.31  $\\
3&$5$&   $23$   &$3.58$&$1.40  $\\
4&$7$&  $71$   &$6.18$&$1.13  $\\
5&$11$&   $311$   &$10.03$&$1.10  $\\
6&$13$&  $479$   &$11.23$&$ 1.16 $\\
7&$17$&  $1559$   &$14.67$&$1.16  $\\
8&$19$&   $5711$   &$18.66$&$1.02  $\\
9&$23$&  $10559$   &$20.63$&$1.11  $\\
10&$29$&  $18191$   &$22.40$&$1.29  $\\
 
11&$31$&  $31391$   &$24.20$&$1.28  $\\ 
12&$37$&  $422231$   &$33.18$&$1.22  $\\
13&$41$&  $701399$   &$40.00$&$1.03  $\\
14&$43$&  $366791$   &$32.68$&$1.32  $\\ 
15&$47$&  $3818929$   &$41.20$&$1.14  $\\   
\end{tabular}
\end{table}

\section{Computational Complexity Of Quadratic Residues And Square Roots} \label{s1154}
The identification of an element $u \in \F_p$ as a quadratic residue (or quadratic nonresidue) has nearly linear deterministic time complexity $O\left (\log p)(\log \log p)^c\right )$ for some $c>0$. The discovery of an algorithm of linear complexity is an open problem, see \cite[p.\ 3]{BS96} for details.\\

The determination of the roots of congruence $x^2-u\equiv 0 \tmod p$ has slightly higher time complexity depending on the following data.
\begin{enumerate} [font=\normalfont, label=(\arabic*)]
\item  The residue class of the prime $p\geq 3$.
\item  The method used to compute a quadratic nonresidue: probabilistic, deterministic.
\item  The method used to compute the roots: probabilistic, deterministic.
\end{enumerate}

The worst case has deterministic time complexity $O\left (\log p)^4\right )$ bit operations. Many algorithms such as Cipolla algorithm, tonelli algorithm, Berlekamp algorithm etc, are explained in \cite[p.\ 157]{BS96}, \cite[p.\ 102]{CP05}, \cite{CN03}. Some polynomials for computing the square roots are provided in \cite{CN11}.

\section{Twin Quadratic Residues And Nonresidues} \label{s5522}
The precise numbers of pairs $QQ$, $QN$, $NQ$ or $NN$ of quadratic residues and quadratic nonresidues are proved in \cite[Theorem 6.3.1]{BW98}. The proof based on Lemma \ref{lem9933.03} is given below.  

\begin{thm} \label{thm5522.02} Let $ p\geq 3$ and $a\in \Z$, $\gcd(a,p)=1$. The number of pairs $n$ and $n+a$ such that $(n|p)=\epsilon_0$ and $(n|p)=\epsilon_1$ is exactly
\begin{equation} \label{eq5522.100}
N(\epsilon_0,\epsilon_1,P)=\frac{1}{4} \left (p-2 - \epsilon_0\left ( \frac{a}{p}\right ) - \epsilon_1 \left ( \frac{-a}{p}\right )-   \epsilon_0\epsilon_1\right ) ,    
\end{equation}
where $\epsilon_i=\pm1$.
\end{thm}
\begin{proof} Sum the product of the two characteristic functions over the finite field $\F_p$:
\begin{eqnarray} \label{eq5522.102}
&& N(\epsilon_0,\epsilon_1,p)\\
&=&\frac{1}{4} \sum_{\substack{n \in \F_p\\ n\ne 0, n\ne-a}}\left (1+\left ( \frac{n}{p}\right )\epsilon_0\right )\left (1+\left ( \frac{n+a}{p}\right )\epsilon_1 \right ) \nonumber\\
&=&\frac{1}{4} \left (p-2+\epsilon_0\sum_{\substack{n \in \F_p\\ n\ne 0, n\ne-a}}\left ( \frac{n}{p}\right )+\epsilon_1\sum_{\substack{n \in \F_p\\ n\ne 0, n\ne-a}}\left ( \frac{n+a}{p}\right ) +\epsilon_0\epsilon_1\sum_{\substack{n \in \F_p\\ n\ne 0, n\ne-a}}\left ( \frac{n^2+an}{p}\right )\right ) \nonumber. 
\end{eqnarray}
Use Lemma \ref{lem7777.25} to evaluate 
\begin{equation} \label{eq5522.104}
\sum_{\substack{n \in \F_p\\ n\ne 0, n\ne-a}}\left ( \frac{n^2+an}{p}\right )=- 1  
\end{equation}
and simplify the expression.
\end{proof}
A new proof based on Lemma \ref{lem9933.03} yields the same result up to a small error term. In particular, at $k=2$, Theorem \ref{thm8833.040} reduces to
\begin{equation} \label{eq5522.110}
N(\epsilon_0,\epsilon_1,P)=N(2,p) =   \frac{p}{2^k}+O\left(k\right )=\frac{p}{4} +O\left (1\right ).    
\end{equation}

\section{Problems} \label{5000}
\subsection{Square Roots Problems}
\begin{exe} \label{exe5000.002} { \normalfont 
Let $n\geq 2$ be a squarefree integer, let $\mathcal{U}=\{u\ne 0: u\cdot u^{-1}\equiv 1 \mod n\}$ be the subset of units (invertible elements), and let $\mathcal{Q}=\{m^2 \tmod n: m\geq 0\}$ be the subset of squares in the finite ring $\Z/n\Z$.  

\begin{enumerate}[font=\normalfont, label=(\alph*)]
\item Show that total number of units is $\#\mathcal{U}=\varphi(n)=n\prod_{p\mid n}(1-1/p)$ units, where $\varphi(n)$ is the number integers relatively prime number to $n$.
\item Show that total number of squares is $\#\mathcal{Q}=\prod_{p\mid n}(p+1)/2$ squares modulo $n$.
\item Show that a square element $s=r^2 \in \Z/n\Z$ has $2^{\omega(n)}$ square roots modulo $n$, where $\omega(n)$ is the number of prime divisors of $n$. For example, $\sqrt{s}=r_0,r_1, \ldots, r_{m-1}$.
\end{enumerate}
}
\end{exe}

\begin{exe} \label{exe5000.004}{ \normalfont 
Let $n\geq 2$ be an integer, and let $\mathcal{Q_0}=\{m^2 \tmod n: m\geq \}-\{0\}$ be the subset of squares. Show that the subset of squares $\#\mathcal{Q} \cup \Z/n\Z$ is a multiplicative subgroup.
}
\end{exe}

\begin{exe} \label{exe5000.006}{ \normalfont 
Let $n\geq 2$ be an integer, and let $\mathcal{Q}=\{m^2 \tmod n: m\geq 1\}$ be the subset of squares. Find a formula for the total number of the subset of squares $\#\mathcal{Q} $.
}
\end{exe}
\begin{exe} \label{exe5000.008}{ \normalfont 
Let $n=pq\geq 6$, where $p\geq 2$ and $q\geq 2$ are distinct primes, and let $m=\varphi(n)$ the number of units in the finite ring $\Z/n\Z$, and let $\mathcal{Q}=\{m^2 \tmod n: m\geq 1\}$ be the subset of squares. Verify these questions:
\begin{enumerate}[font=\normalfont, label=(\alph*)]
\item A square $s=r^2 \in \Z/n\Z$ has $4=2^{\omega(n)}$ square roots $r_0,r_1, r_2, r_{3}$.
\item If $p\equiv q\equiv 3 \tmod 4$, then a single square root $r_i \in \mathcal{Q}$ for some $i=0,1,2,3$.
\end{enumerate}
}
\end{exe}

\begin{exe} \label{exe5000.010}{ \normalfont 
Let $n=pq\geq 6$, where $p\geq 2$ and $q\geq 2$ are distinct primes, and let $m=\varphi(n)$ the number of units in the finite ring $\Z/n\Z$, and let $\mathcal{Q}=\{k^2 \tmod n: k\geq 1\}$ be the subset of squares. Verify these questions:
\begin{enumerate}[font=\normalfont, label=(\alph*)]
\item A square $s=r^2 \in \Z/n\Z$ has $4=2^{\omega(n)}$ square roots $r_0,r_1, r_2, r_{3}$.
\item If $p\equiv q\equiv 1 \tmod 4$, then square roots, $r_i \not \in \mathcal{Q}$ for $i=0,1,2,3$.
\end{enumerate}

}
\end{exe}

\begin{exe} \label{exe5000.014}{ \normalfont 
Let $n\geq 2$ be an integer, and let $\mathcal{Q}=\{m^2 \tmod n: m\geq 1\}$ be the subset of squares. Reference: Handbook of Cryptography.
\begin{enumerate}[font=\normalfont, label=(\alph*)]
\item If $n=pq$, where $p$ and $q$ are primes, and $s\in \mathcal{Q}$, then the inverse $s^{-1}\equiv s^{((p-1)(q-1)+1)/8}\tmod n$.
\item If $n=pqr$, , where $p$, $q$, and $r$ are primes, and $s\in \mathcal{Q}$, find a similar formula for then the inverse $s^{-1}\equiv s^{?((p-1)(q-1)(r-1)+1)/16}\tmod n$.
\end{enumerate}
}
\end{exe}
\subsection{Character Sums Problems}
\begin{exe} \label{exe5000.302} { \normalfont 
Let $q\geq 2$ be an integer, let $\chi$ be a multiplicative nonprincipal character modulo $q$, and $L(s,\chi)=\sum_{n \geq 1}\chi(n)n^{-s}$. Show that $$\sum_{n \leq x} \chi(n)=\frac{1}{i2 \pi} \int_{c-i \infty}^{c+i \infty}L(s,\chi)\frac{x^s}{s} ds,$$
where $c>1$ is a constant, and $x \in \R-\Z$ is a real number.  
}
\end{exe}

\subsection{Sums Of Squares Problems}
\begin{exe} \label{exe5000.382} { \normalfont 
Let $p\equiv 1 \tmod 4$ be a prime, and let $\mathcal{Q}=\{k^2 \tmod p: k\geq 1\}$ be the subset of squares. Show that $$\sum_{r \in \mathcal{Q}} r=\frac{p(p-1)}{4}.$$

}
\end{exe}

\begin{exe} \label{exe5000.388} { \normalfont 
Let $n \geq 1 $ be an integer, and let $\mathcal{Q}=\{k^2 \tmod n: k\geq 1\}$ be the subset of squares. Classify and evaluate the sums of squares $$\sum_{r \in \mathcal{Q}} r?.$$

}
\end{exe}

\subsection{Distribution And Spacing Between Squares Problems}
\begin{exe} \label{exe5000.105}{ \normalfont 
Use the quadratic reciprocity and Dirichlet theorem for primes in arithmetic progressions to prove that the proportion of primes $p$ such that the $n$th prime $p_n$ is the least quadratic nonresidue modulo $p$ has a geometric distribution.  For each $n\geq 1$, the probability  is given by limit $$
 P(n_p=p_n)=\lim_{x\to \infty} \frac{\#\{p\leq x: n_p=p_n\}}{\pi(x)}=\frac{1}{2^n}.
$$ 
}
\end{exe}

\begin{exe} \label{exe5000.105}{ \normalfont 
Use the geometric distribution for the proportion of primes $p$ such that the $n$th prime $p_n$ is the least quadratic nonresidue modulo $p$ to compute the average least quadratic nonresidue $$ \overline{n}=\lim_{x\to \infty} \frac{1}{\pi(x)}\sum_{2<p \leq x}n_p=3.67464 \ldots. 
$$ 
}
\end{exe}
\begin{exe} \label{exe5000.110}{ \normalfont 
Let $n\geq 2$ be a squarefree integer, and let $\mathcal{Q}=\{m^2 \tmod n: m\geq 1\}$ be the subset of squares. Show that the average spacing between squares $$\overline{S_n}=\frac{n}{\#\mathcal{Q}}=\frac{n}{\prod_{p\mid n}(p+1)/2}=\frac{n2^{\omega(n)}}{\psi(n)},$$ 
where $\omega(n)$ is the prime divisors counting function, and $\psi(n)/n=\prod_{p\mid n}(1+1/p)$ is the Dedekind psi function, see Exercise \ref{exe5000.002}.
}
\end{exe}

\begin{exe} \label{exe5000.114}{ \normalfont 
Let $n\geq 2$ be an integer, and let $\mathcal{Q}=\{m^2 \tmod n: m\geq 1\}$ be the subset of squares. Find an expression for the average spacing between squares $$\overline{S_n}=\frac{n}{\#\mathcal{Q}},$$ 
in terms of the $\omega(n)$ is the prime divisors counting function, and $\sigma(n)/n=\prod_{p^v\mid \mid n}(1+1/p+\cdots 1p^v)$ is the sum of divisors function, see Exercise \ref{exe5000.006}.
}
\end{exe}

\subsection{Algorithm Problems}
\begin{exe} \label{5000.800}{ \normalfont 
Let $m,n \in \Z$ be a pair of distinct integers. Construct a deterministic algorithm to compute a simultaneous quadratic nonresidue $\eta$ modulo both $n$ and $m$. Hint: Try a pair of distinct prime $p$ and $q$ first, then generalize it.
}
\end{exe}
\begin{exe} \label{5000.804}{ \normalfont 
Determine the time complexity of computing the inverse $s^{-1} \equiv a \tmod p$ using the Euclidean algorithm. Hint: Consider Lame theorem.
}
\end{exe}

\begin{exe} \label{5000.808}{ \normalfont 
Determine the time complexity of computing the inverse $s^{-1} \equiv a \tmod p$ using the Fermat theorem $s^{p-2} \tmod p$. Hint: Consider the add-multiply algorithm.
}
\end{exe}

\subsection{Primitive Roots Quadratic Nonresidues Problems}
\begin{exe} \label{5000.903}{ \normalfont Given a prime $p\geq 3$, prove the followings.
\begin{enumerate}[font=\normalfont, label=(\alph*)]
\item Show that a primitive root in a finite field $\F_p$ must be a quadratic nonresidue, but a quadratic nonresidue must not be primitive root.
\item A finite field $\F_p$ has $(p-1)/2-\varphi(p-1)$ quadratic nonresidues which are not primitive roots.
\item Verify that every quadratic nonresidues in the finite field $\F_p$ is primitive root if and only if $p=2^{2^n}+1$ is prime.
\end{enumerate}
}
\end{exe}

\subsection{Open Problems}
\begin{exe} \label{5000.203}{ \normalfont  
Given a large prime $p\geq 3$, let $\mathcal{Q}=\{n^2 \tmod p: n\geq 1\}$ be the subset of squares, and let $\mathcal{A},\mathcal{B} \subset \F_p$ be a pair of nonempty subsets. The subsets are proper subsets and must have zero densities in $\F_p$ to avoid trivial cases. Reference: \cite{SI13}, \cite{VS15}.
\begin{enumerate}
\item (Sarkozy conjecture) Prove or disprove the existence of an additive partition $\mathcal{Q}=\mathcal{A}+\mathcal{B}$.
\item Prove or disprove that every square is a sum of two squares: the existence of an additive partition $\mathcal{Q}=\mathcal{A}+\mathcal{B}$, where $\mathcal{A},\mathcal{B} \subset \mathcal{Q}$ are a pair of nonempty proper subsets.
\item Prove or disprove the existence of a difference partition $\mathcal{Q}=\mathcal{A}-\mathcal{B}$. 
\end{enumerate}

}
\end{exe}

\begin{exe} \label{5000.205} { \normalfont 
(Lehmer conjecture) Given a large prime $p\geq 3$, and let $(x\,| \, p)$ be the quadratic symbol, and let $a=r^2$ and $b=s^2$ be a pair of distinct squares. Prove or disprove the claim that 
$$  
 \sum_{n \leq p} \left ( \frac{n+a}{p}\right)\left ( \frac{n}{p}\right)\left ( \frac{n+b}{p}\right)>1
$$ exists for a finite number of primes $p\geq 2$. Reference: \cite[p.\ 246]{GR04}.
}
\end{exe}


\currfilename.\\

\end{document}